\documentclass{article}
\usepackage{amssymb}
\usepackage{amsfonts}

\newtheorem{theorem}{Theorem}

\newtheorem{definition}[theorem]{Definition}

\newtheorem{lemma}[theorem]{Lemma}

\newtheorem{remark}[theorem]{Remark}

\newenvironment{proof}[1][Proof]{\noindent\textbf{#1.} }{\ \rule{0.5em}{0.5em}}
\input{tcilatex}
\begin{document}

\title{Approximating Fixed Points of Nonexpansive Mappings by a Faster
Iteration Process}
\author{Nazli KADIOGLU and Isa YILDIRIM \\
%EndAName
Department of Mathematics, Faculty of Science, Ataturk University\\
nazli.kadioglu@atauni.edu.tr; isayildirim@atauni.edu.tr}
\maketitle

\begin{abstract}
In this paper, we consider a new iteration process which is faster than all
of Picard, Mann, Ishikawa and Agarwal et al. processes. We also prove some
strong and weak convergence theorems for the class of nonexpansive mappings
in Banach spaces.
\end{abstract}

\section{Introduction and Preliminaries}

Throughout this paper, $%
%TCIMACRO{\U{2115} }%
%BeginExpansion
\mathbb{N}
%EndExpansion
$ denotes the set of all positive integers. Let $C$ be a nonempty convex
subset of a normed space $E$, and $T:C\longrightarrow C$ be a mapping. Then
we denote the set of all fixed points of $T$ by $F(T)$. $T$ is called $L-$%
Lipschitzian if there exists a constant $L>0$ such that $\left\Vert
Tx-Ty\right\Vert \leq L\left\Vert x-y\right\Vert $ for all $x,y\in C.$ An $%
L- $Lipschitzian is called contraction if $L\in (0,1)$, and nonexpansive if $%
L=1.$

We know that the Picard \cite{1}, Mann \cite{2} and Ishikawa \cite{3}
iteration processes are defined respectively as: 
\begin{equation}
\left\{ 
\begin{array}{c}
x_{1}=x\in C,\text{ \ \ \ \ \ \ \ \ \ \ \ } \\ 
x_{n+1}=Tx_{n},\text{ \ \ }n\in 
%TCIMACRO{\U{2115} }%
%BeginExpansion
\mathbb{N}
%EndExpansion
,%
\end{array}%
\right.  \tag{1.1}  \label{picard}
\end{equation}%
\begin{equation}
\left\{ 
\begin{array}{c}
x_{1}=x\in C,\text{ \ \ \ \ \ \ \ \ \ \ \ \ \ \ \ \ \ \ \ \ \ \ \ \ \ \ \ \
\ \ \ \ \ \ \ \ \ \ } \\ 
x_{n+1}=(1-\alpha _{n})x_{n}+\alpha _{n}Tx_{n},\text{ \ \ }n\in 
%TCIMACRO{\U{2115} }%
%BeginExpansion
\mathbb{N}
%EndExpansion
,%
\end{array}%
\right.  \tag{1.2}  \label{mann}
\end{equation}

and 
\begin{equation}
\left\{ 
\begin{array}{c}
x_{1}=x\in C,\text{ \ \ \ \ \ \ \ \ \ \ \ \ \ \ \ \ \ \ \ \ \ \ \ \ \ \ \ \
\ \ \ \ \ \ \ \ \ \ } \\ 
x_{n+1}=(1-\alpha _{n})x_{n}+\alpha _{n}Ty_{n},\text{ \ \ \ \ \ \ \ \ \ \ \
\ \ } \\ 
y_{n}=(1-\beta _{n})x_{n}+\beta _{n}Tx_{n},\text{ \ \ }n\in 
%TCIMACRO{\U{2115} }%
%BeginExpansion
\mathbb{N}
%EndExpansion
,\text{ \ \ }%
\end{array}%
\right.  \tag{1.3}  \label{ishikawa}
\end{equation}%
where $\left\{ \alpha _{n}\right\} $ and $\left\{ \beta _{n}\right\} $ are
in $(0,1)$.

Recently, Agarwal, O'Regan and Sahu \cite{4} have introduced the S-iteration
process as follows: 
\begin{equation}
\left\{ 
\begin{array}{c}
a_{1}=a\in C,\text{ \ \ \ \ \ \ \ \ \ \ \ \ \ \ \ \ \ \ \ \ \ \ \ \ \ \ \ \
\ \ \ \ \ \ \ \ \ \ } \\ 
a_{n+1}=(1-\alpha _{n})Ta_{n}+\alpha _{n}Tb_{n},\text{ \ \ \ \ \ \ \ \ \ \ \
\ \ } \\ 
b_{n}=(1-\beta _{n})a_{n}+\beta _{n}Ta_{n},\text{ \ \ }n\in 
%TCIMACRO{\U{2115} }%
%BeginExpansion
\mathbb{N}
%EndExpansion
,\text{ \ \ }%
\end{array}%
\right.  \tag{1.4}  \label{s-iteration}
\end{equation}
where $\left\{ \alpha _{n}\right\} $ and $\left\{ \beta _{n}\right\} $ are
sequences in $(0,1)$.

In \cite{5} motivated by S-iteration process, the first author has
introduced the normal S-iteration process as follows: 
\begin{equation}
\left\{ 
\begin{array}{c}
t_{1}=t\in C,\text{ \ \ \ \ \ \ \ \ \ \ \ \ \ \ \ \ \ \ \ \ \ \ \ \ \ \ \ \
\ \ \ \ \ \ \ \ \ \ \ \ \ \ \ } \\ 
t_{n+1}=T\left( (1-\alpha _{n})t_{n}+\alpha _{n}Tt_{n}\right) ,\text{ \ \ }%
n\in 
%TCIMACRO{\U{2115} }%
%BeginExpansion
\mathbb{N}
%EndExpansion
,%
\end{array}%
\right.  \tag{1.5}  \label{picard-mann}
\end{equation}%
where $\left\{ \alpha _{n}\right\} $ is in $(0,1)$.

In order to compare two fixed point iteration processes $\left\{
u_{n}\right\} $ and $\left\{ v_{n}\right\} $ that converge to a certain
fixed point $p$ of a given operator $T$, Rhoades \cite{6} considered that $%
\left\{ u_{n}\right\} $ is better than $\left\{ v_{n}\right\} $ if 
\[
\left\Vert u_{n}-p\right\Vert \leq \left\Vert v_{n}-p\right\Vert \text{ \ \
for all }n\in 
%TCIMACRO{\U{2115} }%
%BeginExpansion
\mathbb{N}
%EndExpansion
. 
\]

Berinde \cite{7} introduced a different formulation from that of Rhoades as
below:

\begin{definition}
Let $\left\{ a_{n}\right\} $ and $\left\{ b_{n}\right\} $ be two sequences
of positive numbers that converge to $a$, respectively $b$. assume that
there exists 
\[
l=\lim_{n\rightarrow \infty }\frac{\left\vert a_{n}-a\right\vert }{%
\left\vert b_{n}-b\right\vert }. 
\]

\begin{description}
\item[a)] If $l=0$, then it can be said that $\left\{ a_{n}\right\} $
converges to $a$ faster than $\left\{ b_{n}\right\} $ converges to $b$.

\item[b)] If $0<l<\infty $, then it can be said that $\left\{ a_{n}\right\} $
and $\left\{ b_{n}\right\} $ have the same rate of convergence.
\end{description}
\end{definition}

In the sequel, whenever we talk about the rate of convergence, we refer to
the above definition.

Recently, Agarwal et al. \cite{4} showed that, for contractions, S-iteration
process converges at a same rate as Picard iteration and faster than Mann
iteration. Sahu \cite{5} proved that this process converges at a rate faster
than both Picard and Mann iterations for contractions, by giving a numerical
example in support of his claim. After, Khan \cite{8} showed that (\ref%
{picard-mann}) converges at a rate faster than all of Picard (\ref{picard}),
Mann (\ref{mann}) and Ishikawa (\ref{ishikawa}) iterative processes for
contractions.

Our purpose in this paper is to present a new iteration process that, for
contractions, converges faster than both the S-iteration process and the
normal S-iteration process. We also prove a strong convergence theorem with
the help of this process for the class of nonexpansive mappings in general
Banach spaces and apply it to get a result in uniformly convex Banach
spaces. Our iteration process for one mapping is as follows: 
\begin{equation}
\left\{ 
\begin{array}{c}
x_{1}=x\in C, \\ 
x_{n+1}=Ty_{n}, \\ 
y_{n}=(1-\alpha _{n})z_{n}+\alpha _{n}Tz_{n}, \\ 
z_{n}=(1-\beta _{n})x_{n}+\beta _{n}Tx_{n},\text{ }n\in 
%TCIMACRO{\U{2115}}%
%BeginExpansion
\mathbb{N}%
%EndExpansion
\end{array}%
\right.  \tag{1.6}  \label{our iteration}
\end{equation}%
where $\left\{ \alpha _{n}\right\} $ and $\left\{ \beta _{n}\right\} $\ are
in $(0,1)$.

\begin{remark}

\begin{enumerate}
\item[i)] The process (\ref{our iteration}) is indepent of all Picard, Mann,
Ishikawa and S-iteration processes since $\left\{ \alpha _{n}\right\} $ and $%
\left\{ \beta _{n}\right\} $ are in $(0,1)$.

\item[ii)] Even if it is allowed to take $\beta _{n}=0$ and $\alpha
_{n}=\beta _{n}=0$ in the process (\ref{our iteration}), our process reduces
to normal S-iteration (\ref{picard-mann}) and Picard iteration (\ref{picard}%
) processes.
\end{enumerate}
\end{remark}

We recall the following. Let $S=\left\{ x\in E:\left\Vert x\right\Vert
=1\right\} $ and let $E^{\ast }$ be the dual of $E$, that is, the space of
all continuous linear functional $f$ on $E.$ The space $E$ has:

\begin{enumerate}
\item[(i)] \emph{G\^{a}teaux differentiable norm }if 
\[
\lim_{t\rightarrow 0}\frac{\left\Vert x+ty\right\Vert -\left\Vert
x\right\Vert }{t}, 
\]
exists for each $x$ and $y$ in $S$;

\item[(ii)] \emph{Fr\'{e}chet differentiable norm} (see e.g. \cite{9,10}) if
for each $x$ in $S$, the above limit exists and is attained uniformly for $y$
in $S$ and in this case, it is also well-known that 
\begin{equation}
\left\langle h,J(x)\right\rangle +\frac{1}{2}\left\Vert x\right\Vert
^{2}\leq \frac{1}{2}\left\Vert x+h\right\Vert ^{2}\leq \left\langle
h,J(x)\right\rangle +\frac{1}{2}\left\Vert x\right\Vert ^{2}+b\left(
\left\Vert h\right\Vert \right)  \tag{1.7}  \label{frechet}
\end{equation}%
for all $x,$ $h$ in $E$, where $J$ is the Fr\'{e}chet derivative of the
functional $\frac{1}{2}\left\Vert \cdot \right\Vert ^{2}$ at $x\in X$, $%
\left\langle \cdot ,\cdot \right\rangle $ is the pairing between $E$ and $%
E^{\ast }$, and $b$ is an increasing function defined on $[0,\infty )$ such
that $\lim_{t\downarrow 0}\frac{b(t)}{t}=0;$

\item[(iii)] \emph{Opial condition} \cite{11} if for any sequence $\left\{
x_{n}\right\} $ in $E$, $x_{n}\rightharpoonup x$ implies that $\lim
\sup_{n\rightarrow \infty }\left\Vert x_{n}-x\right\Vert <\lim
\sup_{n\rightarrow \infty }\left\Vert x_{n}-y\right\Vert $ for all $y\in E$
with $y\neq x$. Examples of Banach spaces satisfying Opial condition are
Hilbert spaces and all spaces $l^{p}$ $(1<p<\infty )$. On the other hand, $%
L^{p}[0,2\pi ]$ with $1<p\neq 2$ fail to satisfy Opial condition.
\end{enumerate}

A mapping $T:C\longrightarrow E$ is demiclosed at $y\in E$ if for each
sequence $\left\{ x_{n}\right\} $ in $C$ and each and $x\in E$, $%
x_{n}\rightharpoonup x$ and $Tx_{n}\longrightarrow y$ imply that $x\in C$
and $Tx=y.$

\bigskip

We will use the following lemmas in order to prove the our main results.

\begin{lemma}
\label{lemma3}\cite{12} Suppose that $E$ is a uniformly convex Banach space
and $0<p\leq t_{n}\leq q<1$ for all $n\in 
%TCIMACRO{\U{2115} }%
%BeginExpansion
\mathbb{N}
%EndExpansion
$. Let $\left\{ x_{n}\right\} $ and $\left\{ y_{n}\right\} $ be two
sequences of $E$ such that $\lim \sup_{n\rightarrow \infty }\left\Vert
x_{n}\right\Vert \leq r$, $\lim \sup_{n\rightarrow \infty }\left\Vert
y_{n}\right\Vert \leq r$ and $\lim_{n\rightarrow \infty }\left\Vert
t_{n}x_{n}+(1-t_{n})y_{n}\right\Vert =r$ hold for some $r\geq 0$. Then $%
\lim_{n\rightarrow \infty }\left\Vert x_{n}-y_{n}\right\Vert =0.$
\end{lemma}

\begin{lemma}
\label{lemma4}\cite{13} Let $E$ be a uniformly convex Banach space and let $%
C $ be a nonempty closed convex subset of $E$. Let $T$ be a nonexpansive
mapping of $C$ into itself. Then $I-T$ is demiclosed with respect to zero.
\end{lemma}

\section{Rate of Convergence}

In this section, we show that our process (\ref{our iteration}) converges
faster than processes (\ref{s-iteration}) and (\ref{picard-mann}).

\begin{theorem}
Let $C$ be a nonempty closed convex subset of normed space $E$, let $T$ be a
contraction of $C$ into itself. Suppose that each of iterative processes (%
\ref{s-iteration}),(\ref{picard-mann}) and (\ref{our iteration}) converges
to the same fixed point $p$ of $T$ where $\left\{ \alpha _{n}\right\} $ and $%
\left\{ \beta _{n}\right\} $ are such that $0<\lambda \leq \alpha _{n},\beta
_{n}<1$ for all $n\in 
%TCIMACRO{\U{2115} }%
%BeginExpansion
\mathbb{N}
%EndExpansion
$ and for some $\lambda $. Then the iterative process given by (\ref{our
iteration}) converges faster than (\ref{s-iteration}) and (\ref{picard-mann}%
).
\end{theorem}

\begin{proof}
For S-iterative process (\ref{s-iteration}), we have

\begin{eqnarray*}
\left\Vert a_{n+1}-p\right\Vert &=&\left\Vert (1-\alpha _{n})Ta_{n}+\alpha
_{n}Tb_{n}-p\right\Vert \\
&=&\left\Vert (1-\alpha _{n})(Ta_{n}-p)+\alpha _{n}(Tb_{n}-p)\right\Vert \\
&\leq &(1-\alpha _{n})\left\Vert Ta_{n}-p\right\Vert +\alpha _{n}\left\Vert
Tb_{n}-p\right\Vert \\
&\leq &L\left[ (1-\alpha _{n})\left\Vert a_{n}-p\right\Vert +\alpha
_{n}\left\Vert b_{n}-p\right\Vert \right] \\
&=&L\left[ (1-\alpha _{n})\left\Vert a_{n}-p\right\Vert +\alpha
_{n}\left\Vert (1-\beta _{n})a_{n}+\beta _{n}Ta_{n}-p\right\Vert \right] \\
&=&L\left[ (1-\alpha _{n})\left\Vert a_{n}-p\right\Vert +\alpha
_{n}\left\Vert (1-\beta _{n})(a_{n}-p)+\beta _{n}(Ta_{n}-p)\right\Vert %
\right] \\
&\leq &L\left[ (1-\alpha _{n})\left\Vert a_{n}-p\right\Vert +\alpha
_{n}(1-\beta _{n})\left\Vert a_{n}-p\right\Vert +\alpha _{n}\beta
_{n}\left\Vert Ta_{n}-p\right\Vert \right] \\
&\leq &L\left[ (1-\alpha _{n})+\alpha _{n}(1-\beta _{n})+L\alpha _{n}\beta
_{n}\right] \left\Vert a_{n}-p\right\Vert \\
&=&L\left( 1-\alpha _{n}\beta _{n}(1-L)\right) \left\Vert a_{n}-p\right\Vert
\\
&\leq &L\left( 1-\lambda ^{2}(1-L)\right) \left\Vert a_{n}-p\right\Vert \\
&&\vdots \\
&\leq &\left[ L\left( 1-\lambda ^{2}(1-L)\right) \right] ^{n}\left\Vert
a_{1}-p\right\Vert .
\end{eqnarray*}

Let $k_{n}=\left[ L\left( 1-\lambda ^{2}(1-L)\right) \right] ^{n}\left\Vert
a_{1}-p\right\Vert .$

From (\ref{picard-mann}), we obtain that

\begin{eqnarray*}
\left\Vert t_{n+1}-p\right\Vert &=&\left\Vert T\left( (1-\alpha
_{n})t_{n}+\alpha _{n}Tt_{n}\right) -p\right\Vert \\
&\leq &L\left\Vert (1-\alpha _{n})(t_{n}-p)+\alpha _{n}(Tt_{n}-p)\right\Vert
\\
&\leq &L\left[ (1-\alpha _{n})\left\Vert t_{n}-p\right\Vert +\alpha
_{n}L\left\Vert t_{n}-p\right\Vert \right] \\
&=&L\left( 1-(1-L)\alpha _{n}\right) \left\Vert t_{n}-p\right\Vert \\
&\leq &L\left( 1-(1-L)\lambda \right) \left\Vert t_{n}-p\right\Vert \\
&&\vdots \\
&\leq &\left[ L\left( 1-(1-L)\lambda \right) \right] ^{n}\left\Vert
t_{1}-p\right\Vert .
\end{eqnarray*}

Let $l_{n}=\left[ L\left( 1-(1-L)\lambda \right) \right] ^{n}\left\Vert
t_{1}-p\right\Vert .$

Our process (\ref{our iteration}) gives

\begin{eqnarray*}
\left\Vert x_{n+1}-p\right\Vert &=&\left\Vert Ty_{n}-p\right\Vert \\
&\leq &L\left\Vert y_{n}-p\right\Vert \\
&=&L\left\Vert (1-\alpha _{n})z_{n}+\alpha _{n}Tz_{n}-p\right\Vert \\
&=&L\left\Vert (1-\alpha _{n})(z_{n}-p)+\alpha _{n}(Tz_{n}-p)\right\Vert \\
&\leq &L\left[ (1-\alpha _{n})\left\Vert z_{n}-p\right\Vert +\alpha
_{n}\left\Vert Tz_{n}-p\right\Vert \right] \\
&\leq &L\left[ (1-\alpha _{n})\left\Vert z_{n}-p\right\Vert +\alpha
_{n}L\left\Vert z_{n}-p\right\Vert \right] \\
&=&L(1-(1-L)\alpha _{n})\left\Vert z_{n}-p\right\Vert \\
&=&L(1-(1-L)\alpha _{n})\left\Vert (1-\beta _{n})x_{n}+\beta
_{n}Tx_{n}-p\right\Vert \\
&=&L(1-(1-L)\alpha _{n})\left\Vert (1-\beta _{n})(x_{n}-p)+\beta
_{n}(Tx_{n}-p)\right\Vert \\
&\leq &L(1-(1-L)\alpha _{n})\left[ (1-\beta _{n})\left\Vert
x_{n}-p\right\Vert +\beta _{n}\left\Vert Tx_{n}-p\right\Vert \right] \\
&\leq &L(1-(1-L)\alpha _{n})\left[ (1-\beta _{n})\left\Vert
x_{n}-p\right\Vert +\beta _{n}L\left\Vert x_{n}-p\right\Vert \right] \\
&=&L(1-(1-L)\alpha _{n})(1-(1-L)\beta _{n})\left\Vert x_{n}-p\right\Vert \\
&\leq &L(1-(1-L)\lambda )^{2}\left\Vert x_{n}-p\right\Vert \\
&&\vdots \\
&\leq &\left[ L(1-(1-L)\lambda )^{2}\right] ^{n}\left\Vert
x_{1}-p\right\Vert .
\end{eqnarray*}%
Let $m_{n}=\left[ L(1-(1-L)\lambda )^{2}\right] ^{n}\left\Vert
x_{1}-p\right\Vert .$Then 
\[
\frac{m_{n}}{k_{n}}=\frac{\left[ L(1-(1-L)\lambda )^{2}\right]
^{n}\left\Vert x_{1}-p\right\Vert }{\left[ L\left( 1-(1-L)\lambda
^{2}\right) \right] ^{n}\left\Vert a_{1}-p\right\Vert }=\left[ \frac{%
(1-(1-L)\lambda )^{2}}{1-(1-L)\lambda ^{2}}\right] ^{n}\frac{\left\Vert
x_{1}-p\right\Vert }{\left\Vert a_{1}-p\right\Vert }\longrightarrow 0\text{
as }n\longrightarrow \infty . 
\]%
Thus $\left\{ x_{n}\right\} $ converges faster than $\left\{ a_{n}\right\} $
to $p.$ Similarly 
\[
\frac{m_{n}}{l_{n}}=\frac{\left[ L(1-(1-L)\lambda )^{2}\right]
^{n}\left\Vert x_{1}-p\right\Vert }{\left[ L\left( 1-(1-L)\lambda \right) %
\right] ^{n}\left\Vert t_{1}-p\right\Vert }=\left[ 1-(1-L)\lambda \right]
^{2n}\frac{\left\Vert x_{1}-p\right\Vert }{\left\Vert t_{1}-p\right\Vert }%
\longrightarrow 0\text{ as }n\longrightarrow \infty . 
\]%
Hence $\left\{ x_{n}\right\} $ converges faster than $\left\{ t_{n}\right\} $
to $p.$
\end{proof}

\section{Convergence Theorems}

In this section, we give some convergence theorems using our iteration
process (\ref{our iteration}).

\begin{lemma}
\label{lemma6}Let $C$ be a nonempty closed convex subset of a uniformly
convex Banach space $E$ and let $T$ be a nonexpansive self mapping of $C$.
Let $\left\{ x_{n}\right\} $ be defined by the iteration process (\ref{our
iteration}) where $\left\{ \alpha _{n}\right\} $ and $\left\{ \beta
_{n}\right\} $ are in $(0,1)$ for all $n\in 
%TCIMACRO{\U{2115} }%
%BeginExpansion
\mathbb{N}
%EndExpansion
$. Then

\begin{enumerate}
\item[(i)] $\lim_{n\rightarrow \infty }\left\Vert x_{n}-p\right\Vert $
exists for all $p\in F(T).$

\item[(ii)] $\lim_{n\rightarrow \infty }\left\Vert x_{n}-Tx_{n}\right\Vert
=0.$
\end{enumerate}
\end{lemma}

\begin{proof}
Let $p\in F(T).$ Then

\begin{eqnarray}
\left\Vert z_{n}-p\right\Vert &=&\left\Vert (1-\beta _{n})x_{n}+\beta
_{n}Tx_{n}-p\right\Vert  \TCItag{1.8}  \label{yn..xn} \\
&=&\left\Vert (1-\beta _{n})(x_{n}-p)+\beta _{n}(Tx_{n}-p)\right\Vert 
\nonumber \\
&\leq &(1-\beta _{n})\left\Vert x_{n}-p\right\Vert +\beta _{n}\left\Vert
Tx_{n}-p\right\Vert  \nonumber \\
&\leq &(1-\beta _{n})\left\Vert x_{n}-p\right\Vert +\beta _{n}\left\Vert
x_{n}-p\right\Vert  \nonumber \\
&=&\left\Vert x_{n}-p\right\Vert ,  \nonumber
\end{eqnarray}%
and so

\begin{eqnarray*}
\left\Vert x_{n+1}-p\right\Vert &=&\left\Vert Ty_{n}-p\right\Vert \\
&\leq &\left\Vert y_{n}-p\right\Vert \\
&=&\left\Vert (1-\alpha _{n})z_{n}+\alpha _{n}Tz_{n}-p\right\Vert \\
&=&\left\Vert (1-\alpha _{n})(z_{n}-p)+\alpha _{n}(Tz_{n}-p)\right\Vert \\
&\leq &(1-\alpha _{n})\left\Vert z_{n}-p\right\Vert +\alpha _{n}\left\Vert
Tz_{n}-p\right\Vert \\
&\leq &(1-\alpha _{n})\left\Vert z_{n}-p\right\Vert +\alpha _{n}\left\Vert
z_{n}-p\right\Vert \\
&=&\left\Vert z_{n}-p\right\Vert \\
&\leq &\left\Vert x_{n}-p\right\Vert .
\end{eqnarray*}%
This shows that $\left\{ \left\Vert x_{n}-p\right\Vert \right\} $ is
decreasing, and this proves part (i). Let 
\begin{equation}
\lim_{n\rightarrow \infty }\left\Vert x_{n}-p\right\Vert =c.  \tag{1.9}
\label{xn=c}
\end{equation}%
Now, $\left\Vert Tx_{n}-p\right\Vert \leq \left\Vert x_{n}-p\right\Vert $
implies that 
\begin{equation}
\lim \sup_{n\rightarrow \infty }\left\Vert Tx_{n}-p\right\Vert \leq c. 
\tag{1.10}  \label{txn..c}
\end{equation}

Since $\left\Vert x_{n+1}-p\right\Vert \leq \left\Vert z_{n}-p\right\Vert $,
therefore 
\[
\lim \inf_{n\rightarrow \infty }\left\Vert x_{n+1}-p\right\Vert \leq \lim
\inf_{n\rightarrow \infty }\left\Vert z_{n}-p\right\Vert , 
\]%
so 
\begin{equation}
c\leq \lim \inf_{n\rightarrow \infty }\left\Vert z_{n}-p\right\Vert 
\tag{1.11}  \label{c..yn}
\end{equation}%
On the other hand, (\ref{yn..xn}) implies that 
\begin{equation}
\lim \sup_{n\rightarrow \infty }\left\Vert z_{n}-p\right\Vert \leq c. 
\tag{1.12}  \label{yn=c}
\end{equation}%
From (\ref{c..yn}) and (\ref{yn=c}) 
\[
\lim_{n\rightarrow \infty }\left\Vert z_{n}-p\right\Vert =c. 
\]%
Hence, this implies that 
\begin{equation}
c=\lim_{n\rightarrow \infty }\left\Vert z_{n}-p\right\Vert
=\lim_{n\rightarrow \infty }\left\Vert (1-\beta _{n})(x_{n}-p)+\beta
_{n}(Tx_{n}-p)\right\Vert .  \tag{1.13}  \label{xn+txn=c}
\end{equation}%
Using (\ref{xn=c}), (\ref{txn..c}), (\ref{xn+txn=c}) and Lemma \ref{lemma3},
we obtain 
\[
\lim_{n\rightarrow \infty }\left\Vert x_{n}-Tx_{n}\right\Vert =0. 
\]
\end{proof}

\begin{lemma}
\label{lemma7}Assume that all the conditions of Lemma \ref{lemma6} are
satisfied. Then, for any $p_{1},$ $p_{2}\in F(T)$, $\lim_{n\rightarrow
\infty }\left\langle x_{n},J(p_{1}-p_{2})\right\rangle $ exists; in
particular, $\left\langle p-q,J(p_{1}-p_{2})\right\rangle =0$ for all $%
p,q\in \omega _{w}(x_{n})$, the set of all weak limits of $\left\{
x_{n}\right\} .$
\end{lemma}

\begin{proof}
The proof of this lemma is the same as the proof of Lemma 2.3 of \cite{14}.
So, we omit it here.
\end{proof}

\bigskip

We now give our weak convergence theorem.

\begin{theorem}
Let $E$ be a uniformly convex Banach space and let $C,$ $T$ and $\left\{
x_{n}\right\} $ be taken as in Lemma \ref{lemma6}. Assume that $(a)$ $E$
satifies Opial's condition or $(b)$ $E$ has a Fr\'{e}chet differentiable
norm. If $F(T)\neq \varnothing ,$ then $\left\{ x_{n}\right\} $ converges
weakly to a fixed point of $T$.
\end{theorem}

\begin{proof}
From \emph{(i)} in Lemma \ref{lemma6}, we know that $\lim_{n\rightarrow
\infty }\left\Vert x_{n}-p\right\Vert $ exists \ for all $p\in F(T)$. Thus $%
\left\{ x_{n}\right\} $ is bounded. Since $E$ is uniformly convex, $\left\{
x_{n}\right\} $ has a subsequence $\left\{ x_{n_{k}}\right\} $ which
converges weakly in $C$. We prove that $\left\{ x_{n}\right\} $ has a uniqe
weak subsequential limit in $F(T)$. For this, let $u$ and $v$ be weak limits
of subsequences $\left\{ x_{n_{i}}\right\} $ and $\left\{ x_{n_{j}}\right\} $
of $\left\{ x_{n}\right\} $, respectively. By Lemma \ref{lemma6}, $%
\lim_{n\rightarrow \infty }\left\Vert x_{n}-Tx_{n}\right\Vert =0$ and $I-T$
is demiclosed with respect to zero by Lemma \ref{lemma4}; therefore, we
obtain $Tu=u.$ Again, in the same manner, we can prove that $v\in F(T).$
Next, we prove the uniqueness. To this end, first assume $(a)$ is true. If $%
u $ and $v$ are distinct, then by Opial's condition, 
\begin{eqnarray*}
\lim_{n\rightarrow \infty }\left\Vert x_{n}-u\right\Vert
&=&\lim_{n_{i}\rightarrow \infty }\left\Vert x_{n_{i}}-u\right\Vert \\
&<&\lim_{n_{i}\rightarrow \infty }\left\Vert x_{n_{i}}-v\right\Vert \\
&=&\lim_{n\rightarrow \infty }\left\Vert x_{n}-v\right\Vert \\
&=&\lim_{n_{j}\rightarrow \infty }\left\Vert x_{n_{j}}-v\right\Vert \\
&<&\lim_{n_{j}\rightarrow \infty }\left\Vert x_{n_{j}}-u\right\Vert \\
&=&\lim_{n\rightarrow \infty }\left\Vert x_{n}-u\right\Vert .
\end{eqnarray*}%
This is a contradiction, so $u=v$. Next assume $(b)$. By Lemma \ref{lemma7}, 
$\left\langle p-q,J(p_{1}-p_{2})\right\rangle =0$ for all $p,q\in \omega
_{w}(x_{n})$. Therefore $\left\Vert u-v\right\Vert ^{2}=\left\langle
u-v,J(u_{1}-v_{2})\right\rangle =0$ implies $u=v$. Consequently, $\left\{
x_{n}\right\} $ converges weakly to a point of $F$ and this completes the
proof.
\end{proof}

\bigskip

A mapping $T:C\longrightarrow C$, where $C$ is a subset of normed space $E$,
is said to satisfy Condition (A) \cite{15} if there exists a nondecrerasing
function $f:[0,\infty )\longrightarrow \lbrack 0,\infty )$ with $f(0)=0$, $%
f(r)>0$ for all $r\in (0,\infty )$ such that $\left\Vert x-Tx\right\Vert
\geq f(d(x,F(T)))$ for all $x\in C$, where $d(x,F(T))=\inf \left\{
\left\Vert x-p\right\Vert :p\in F(T)\right\} .$

\begin{theorem}
\label{teorem9}Let $E$ be a uniformly convex Banach space and let $C,$ $T$
and $\left\{ x_{n}\right\} $ be taken as in Lemma \ref{lemma6}. Then $%
\left\{ x_{n}\right\} $ converges to a point of $F(T)$ if and only if $\lim
\inf_{n\rightarrow \infty }d(x_{n},F(T))=0$.
\end{theorem}

\begin{proof}
Necessity is obvious. Suppose that $\lim \inf_{n\rightarrow \infty
}d(x_{n},F(T))=0.$ From \emph{(i)} in Lemma \ref{lemma6}, we know that $%
\lim_{n\rightarrow \infty }\left\Vert x_{n}-p\right\Vert $ exists for all $%
p\in F(T)$, therefore $\lim_{n\rightarrow \infty }d(x_{n},F(T))$ exists. But
by hypothesis, $\lim \inf_{n\rightarrow \infty }d(x_{n},F(T))=0$, therefore
we have $\lim_{n\rightarrow \infty }d(x_{n},F(T))=0$. We will show that $%
\left\{ x_{n}\right\} $ is a Cauchy sequence in $C$. Since $%
\lim_{n\rightarrow \infty }d(x_{n},F(T))=0$, for given $\varepsilon >0$,
there exists $n_{0}$ in $%
%TCIMACRO{\U{2115} }%
%BeginExpansion
\mathbb{N}
%EndExpansion
$ such that for all $n\geq n_{0}$, 
\[
d(x_{n},F(T))<\frac{\varepsilon }{2}. 
\]%
Particularly, $\inf \left\{ \left\Vert x_{n_{0}}-p\right\Vert :p\in
F(T)\right\} <\frac{\varepsilon }{2}.$ Hence, there exists $p^{\ast }\in
F(T) $ such that $\left\Vert x_{n_{0}}-p^{\ast }\right\Vert <\frac{%
\varepsilon }{2}$. Now, for $m,n\geq n_{0}$, 
\[
\left\Vert x_{n+m}-x_{n}\right\Vert \leq \left\Vert x_{n+m}-p^{\ast
}\right\Vert +\left\Vert x_{n}-p^{\ast }\right\Vert \leq 2\left\Vert
x_{n_{0}}-p^{\ast }\right\Vert <\varepsilon . 
\]%
Hence $\left\{ x_{n}\right\} $ is a Cauchy sequence in $C$. Since $C$ is
closed in the Banach space $E$, there exists a point $q$ in $C$ such that $%
\lim_{n\rightarrow \infty }x_{n}=q$. Now $\lim_{n\rightarrow \infty
}d(x_{n},F(T))=0$ gives that $d(q,F(T))=0$. Since $F$ is closed, $q\in F(T)$.
\end{proof}

\bigskip

Note that this condition is weaker than the requirement that $T$ is
demicompact or $C$ is compact, see \cite{15}. Applying Theorem \ref{teorem9}%
, we obtain a strong convergence of the process (\ref{our iteration}) under
Condition (A) as follows.

\begin{theorem}
\label{teorem10}Let $E$ be a uniformly convex Banach space and let $C,$ $T$
and $\left\{ x_{n}\right\} $ be as in Lemma \ref{lemma6}. If $T$ satisfies
Condition (A), then $\left\{ x_{n}\right\} $ converges strongly to a fixed
point of $T$.
\end{theorem}

\begin{proof}
From Lemma \ref{lemma6}, we know that 
\begin{equation}
\lim_{n\rightarrow \infty }\left\Vert x_{n}-Tx_{n}\right\Vert =0.  \tag{1.14}
\label{lim=0}
\end{equation}%
Using Condition (A) and (\ref{lim=0}), we get 
\[
\lim_{n\rightarrow \infty }f(d(x_{n},F(T)))\leq \lim_{n\rightarrow \infty
}\left\Vert x_{n}-Tx_{n}\right\Vert =0, 
\]%
That is, $\lim_{n\rightarrow \infty }f(d(x_{n},F(T)))=0$. Since $f:[0,\infty
)\longrightarrow \lbrack 0,\infty )$ is a nondecrerasing function satisfying 
$f(0)=0$, $f(r)>0$ for all $r\in (0,\infty )$, therefore, we have 
\[
\lim_{n\rightarrow \infty }d(x_{n},F(T))=0. 
\]%
Now all the conditions of Theorem \ref{teorem9}\ are satisfied, therefore,
by its conclusion, $\left\{ x_{n}\right\} $ converges strongly to a fixed
point of $T$.
\end{proof}

\end{document}